\documentclass[11pt,a4paper]{article}
\usepackage[T1]{fontenc}
\usepackage[utf8]{inputenc}
\usepackage{lmodern}
\usepackage[margin=25mm]{geometry}
\usepackage{amsmath,amssymb,amsthm,booktabs,array,tikz,microtype}

\usepackage{graphicx,placeins}

\usetikzlibrary{arrows.meta,calc}
\usepackage[colorlinks=true,linkcolor=blue!40!black,citecolor=blue!40!black,urlcolor=blue!40!black]{hyperref}
\hypersetup{pdftitle={Peeling sequences: a directional method for the three-block construction},pdfauthor={Daniel Gabor Simon}}
\newtheorem{theorem}{Theorem}
\newtheorem{lemma}[theorem]{Lemma}

\newcommand{\alp}{\alpha}
\title{Peeling sequences: a directional method for the three-block construction}
\author{D\'aniel G\'abor Simon\thanks{Corresponding author. E-mail: \href{mailto:dgsimon@renyi.hu}{\nolinkurl{dgsimon@renyi.hu}}, ORCID: \href{https://orcid.org/0000-0002-3750-9666}{\nolinkurl{0000-0002-3750-9666}}.\protect\\
Research was supported by the ERC Advanced Grant ``Geoscape''.\protect\\
\textbf{Acknowledgement.} The inspiration for the proof can be traced back to a discussion I had with Géza Tóth many years ago, so I would like to thank him for that.\protect\\
\textbf{AI assistance.} I used OpenAI's ChatGPT-5.6 Sol, ChatGPT-6 Astra, and ChatGPT-6.1 Sol to assist with developing and checking the proofs, preparing the figures, and revising the exposition. I have checked the mathematical arguments and take full
responsibility for the results in the final manuscript.}\\[3pt]
\small HUN-REN Alfr\'ed R\'enyi Institute of Mathematics, Budapest, Hungary\\
\small E\"otv\"os Lor\'and University, Budapest, Hungary}
\date{\today}
\begin{document}
\maketitle

\begin{abstract}
A \emph{peeling sequence} of a finite planar point set is an ordering of its points, such that when the points get removed in this order,
 each removed point is a vertex of the current convex hull of the set.
Write $g(S)$ for the number of such sequences, and $g(n)$ for the minimum
of $g(S)$ over $n$-point sets in general position. We present a method which can be used to prove better upper bounds on the previously analysed recursive 3-block constructions $S_n$. In fact, we prove $g(n)\leq g(S_n)=O(6.57^n)$, using directional restrictions and a weighted prefix-tree form of the
Kraft--McMillan inequality. From the other direction, we show that the same construction cannot attain the current best known lower bound of $c\cdot3^n$. In fact, $g(S_n)\geq 5.22^n$ when $n$ is a large enough power of $3$. Improving the general lower bound on $g(n)$ is still an open question.
\end{abstract}

\noindent\textbf{Keywords.} Peeling sequences; convex hull; recursive point
sets; geometric counting.\\
\textbf{2020 Mathematics Subject Classification.} 52C35, 05A16.

\section{Introduction}
A \emph{peeling sequence} of a finite planar point set $S$ is an ordering
$(p_1,\ldots,p_n)$ of its points such that $p_i$ is a vertex of
$\operatorname{conv}\{p_i,\ldots,p_n\}$ for every $i$. Write $g(S)$ for the
number of peeling sequences of $S$, and put
\[
 g(n)=\min\{g(S):S\subset\mathbb R^2,\ |S|=n,
                  \ S\text{ is in general position}\}.
\]
Here general position means that no three points are collinear. If $S$ is in convex position, it has exactly $g(S)=n!$ since points can be peeled in any order. The question of minimizing the number of such sequences was introduced by
Dumitrescu~\cite{Dum22}.

The removal of convex-hull vertices has a longer history in the study
of convex layers. In that process, all vertices of the current convex
hull are removed at the same time, and this is repeated on the remaining
point set. Convex layers have been studied in statistics
\cite{Barnett76} and computational geometry \cite{Chazelle85}.
Chazelle \cite{Chazelle85} gave an optimal $O(n\log n)$ algorithm for
computing them, and Har-Peled and Lidick\'y \cite{HL13} studied the
number of layers of the integer grid. In our problem, only one vertex
is removed at each step, and we count all possible removal orders.

Note that every set of $n\geq 3$ points in general position has at least $3$ points on its convex hull, proving the recursive lower bound $g(n)\geq 3g(n-1)$ for $n\geq 3$. Based on that, the easy lower-bound $g(n)\geq2\cdot 3^{n-2}$ follows, which is still the best known lower-bound up to an improved constant factor. 

Dumitrescu~\cite{Dum22} obtained a subfactorial upper bound
$2^{O(n\log\log n)}$. Dumitrescu and T\'oth~\cite{DT25} subsequently
introduced a recursive construction giving $g(n)=O(12.29^n)$.
A more detailed analysis of this construction improved the upper bound
to $O(9.78^n)$~\cite{Simon26}. Recently, Hisatsuga, Johnston and
Miyazaki~\cite{HJM26} proved
\[
 g(n)\le(8+o(1))^n
\]
by combining estimates for one-sided truncations with a weighted
prefix-tree argument.

A related recursive construction was used by Edelsbrunner and
Welzl \cite{EW85} to obtain point sets with many halving lines;
this connection was already noted by Dumitrescu and T\'oth
\cite[Remark~3.1]{DT25}. The minimum number of peeling sequences has also
been studied in higher dimensions \cite{DT25,SimonHigher26}.

Our final counting argument also uses a weighted form of the
Kraft--McMillan inequality \cite{Kraft49,McMillan56,CT06},
as in \cite[Lemma~1.2]{HJM26}. The additional information we keep is the
directional restriction on each surviving block, together with the
sign of its preceding endpoint deletions.

We use the recursive three-block construction of Dumitrescu and T\'oth \cite{DT25}, while adding orientation to the affine copies. This idea of directing the point sets is the key to further improving the bounds. While all
three blocks remain, each legal deletion is forced within each
block from outside to inside. After the first block disappears, the surviving blocks are subject
to some directional restrictions. This is because the presence of the other block essentially forbids peeling from one side of each block. Formalizing these restrictions,
together with the signs of the preceding endpoint deletions, gives a
family of constrained peeling problems closed under recursion. Using weights chosen from one level of the construction, and applying the resulting estimate recursively at every level, we get the following theorem:

\begin{theorem}\label{thm:main}
Let $\beta=\frac{9+\sqrt{17}}2$, $\alp=\log_3 2.$ There is a recursively defined family of $n$-point sets $S_n$ in general
position such that, for every $n\ge1$,
\[
 g(n)\le g(S_n)
 \le\beta^n48^{\,2^{\lceil\log_3 n\rceil}-1}
 =\beta^n\exp\bigl(O(n^\alp)\bigr).
\]
In particular,
\[
 \limsup_{n\to\infty}g(S_n)^{1/n}\le\beta<6.562,
 \qquad g(n)=O(6.57^n).
\]
\end{theorem}

While we do not improve the general lower bound on $g(n)$, we can show a lower bound on the number of peeling sequences our $3$-block recursive construction can attain.

\begin{theorem}\label{thm:lowerbound}
Let $S_n$ be the recursive $3$-block construction. For every sufficiently large $k$, if $n=3^k$ then
\[
g(S_n)\geq 5.22^n.
\]
\end{theorem}

Sections~\ref{sec:construction}--\ref{sec:counting} contain the proof of
Theorem~\ref{thm:main}, with the description of the recursive construction $S_n$ and coordinate verification needed for Lemma~\ref{lem:router} in
Appendix~\ref{sec:coordinates}. Section \ref{sec:lower} contains the proof of Theorem \ref{thm:lowerbound}.

\section{The recursive construction and its deletion commands}\label{sec:construction}
For a finite point set with distinct first coordinates,   we define $5$ different deletion commands. Let $A$ mean
``delete any convex-hull vertex'', and let $B,C$ mean ``delete an upwards visible vertex of
the convex hull'' and ``delete a downwards visible vertex of
the convex hull'' respectively. Both commands include the leftmost and rightmost points. A vertex $v_0=(x_0,y_0)$ in a set $S$ is upwards visible, if for any $v=(x_0,y)\in \operatorname{conv}(S)$, $y\leq y_0$. Downwards visibility is defined similarly. Let $E_+,E_-$ mean ``delete the rightmost,
respectively leftmost, point''. Directions inside a block are always
measured in that block's own internal coordinates. The idea behind commands $B$ and $C$ is that whenever $2$ blocks are alive, for each block the points from the other one prevent us removing from one side of our block. Hence we define for each flattened block an internal direction, and based on that direction every block is going to have an upwards and a downwards side. A more detailed explanation can be found in Appendix \ref{sec:coordinates}.

\begin{lemma}\label{lem:router}
There is a family $S_n$ in general position with distinct first coordinates,
with $S_1$ a single point, in which, for $n\ge2$, $S_n$ is the
union of three thin affine copies $L,M,R$ of sizes
\[
 n_L=\lfloor n/3\rfloor,\quad n_M=\lfloor(n+1)/3\rfloor,
 \quad n_R=\lfloor(n+2)/3\rfloor.
\]
Empty copies are omitted. The following rules hold after arbitrary deletions from the blocks, with all commands applied to the points which remain.

When all three blocks are nonempty, commands $A$ and $C$ offer the $E_+$ point
of each block, and $B$ offers the $E_+$ points of $L,R$ only.
When two non-empty blocks remain, the instructions restricted to them are
\[
\begin{array}{c|ccc}
\text{blocks}&A&B&C\\\hline
LM&(B,B)&(B,E_-)&(E_+,B)\\
LR&(C,C)&(E_+,E_+)&(C,C)\\
MR&(C,B)&(C,B)&(E_+,E_+).
\end{array}
\]
With one remaining block, that block inherits command $A$ if the original command was $A$. Commands $B,C$ retain their
names in $L,R$ and interchange names in $M$. Endpoint commands are
\[
\begin{array}{c|c|c}
\text{command}&\text{block order}&\text{local endpoint commands}\\\hline
 E_+&R,M,L&E_+,E_-,E_-\\
 E_-&L,M,R&E_+,E_+,E_-.
\end{array}
\]
\end{lemma}
The rules are exact: a point can be deleted precisely when the indicated
local command allows it. The proof of the lemma is in
Appendix~\ref{sec:coordinates}.

\section{Defining command strings and setting up the recursion}\label{sec:responses}
For $\sigma\in\{+,-\}$, $D\in\{B,C\}$ and integers $p,q\ge0$ with $p+q\le n$, we define
\[
 F_n^{\sigma,D}(p,q)
\]
to be the number of peeling sequences of $S_n$ following command $E_\sigma$ for $p$ deletions, then command $D$ for $q$
deletions, then command $A$ for all remaining deletions. Note that $p$ and/or $q$ might be zero, in which case the sign and respectively, the direction we choose is arbitrary.
In particular,
$F_n^{\sigma,D}(0,0)=g(S_n)$. Similarly, $F_n^{\sigma,D}(n,0)=1$ since if all commands are just removing endpoints from the same direction, the peeling sequence is determined. The recursive arguments below concern $n\ge3$, when all three
blocks are nonempty. The cases $n=1,2$ will be checked directly.

\begin{lemma}\label{lem:closure}
Let $n\ge3$. For any block $X\in\{L,M,R\}$ of $S_n$, and any peeling sequence $\pi$ counted by $F_n^{\sigma,D}(p,q)$, consider the commands inherited by $X$ in the steps where $\pi$ removes an element of $X$. The list of these inherited commands:
\begin{itemize}
    \item Starts with a prefix using only endpoint commands $E_{\sigma'}$ of the same sign.
    \item Continues with either a string of command $B$'s or a string of command $C$'s.
    \item Finishes with a string of command $A$'s.
\end{itemize}
Each of the three phases might have length $0$, in which case the phase is skipped. The first block emptied inherits the string of only $E_+$ commands.
\end{lemma}

\begin{proof}
For the first block emptied, the statement is obvious, since it only
inherits $E_+$ commands.

For the other blocks, we use the command inheritance patterns from
Lemma~\ref{lem:router}. Let $X$ be the block whose inherited commands we
investigate. The proof consists of a careful case study, which we
visualize. There are three stages based on how many blocks are still
alive, and three stages based on which stage of the main command string
we are at. We analyze what command $X$ inherits for each pair of these
stages, and organize them in $3\times3$ grid diagrams in
Figures~\ref{fig:L}, \ref{fig:M}, and \ref{fig:R}. The rows correspond to
$3,2,1$ blocks remaining, from bottom to top, and the columns correspond
to the main commands $E_\sigma,D,A$, from left to right. The choice of
which grid to analyze is based on the command $D$ used in the directional
stage and on the first disappearing block. In the first column, we use the
red or blue entries corresponding to the sign of the main endpoint command.
Red corresponds to $E_+$ and blue to $E_-$. An entry $\varnothing$
means that no point of $X$ can be removed in that case: either the stage
cannot occur with the specified first disappearing block, or the main
command does not allow a deletion from $X$.

The crucial observation is the following: as the command string
progresses, our position in the grid can only move towards the right or
the top. This is because we cannot return to an earlier stage of the
main command string, and the number of blocks still alive cannot
increase. For the inherited string, we just record the command in the current square whenever a point
of $X$ is removed. The same entry might be repeated, while some
squares might contribute no commands.

Upon inspection of all the grids, one can verify that along any such
path which can arise from a peeling sequence, the inherited endpoint commands have the same sign, followed by
only $B$'s or only $C$'s, and finally by only $A$'s. Any of these stages might
be empty. In particular, after an inherited $B$ or $C$, every later
nonempty entry on such a path is the same directional command or $A$;
after an inherited $A$, only $A$ can follow. Hence the inherited string has the form specified in the lemma.
\end{proof}

\begin{figure}[!htbp]
    \centering
    \includegraphics[width=0.72\linewidth]{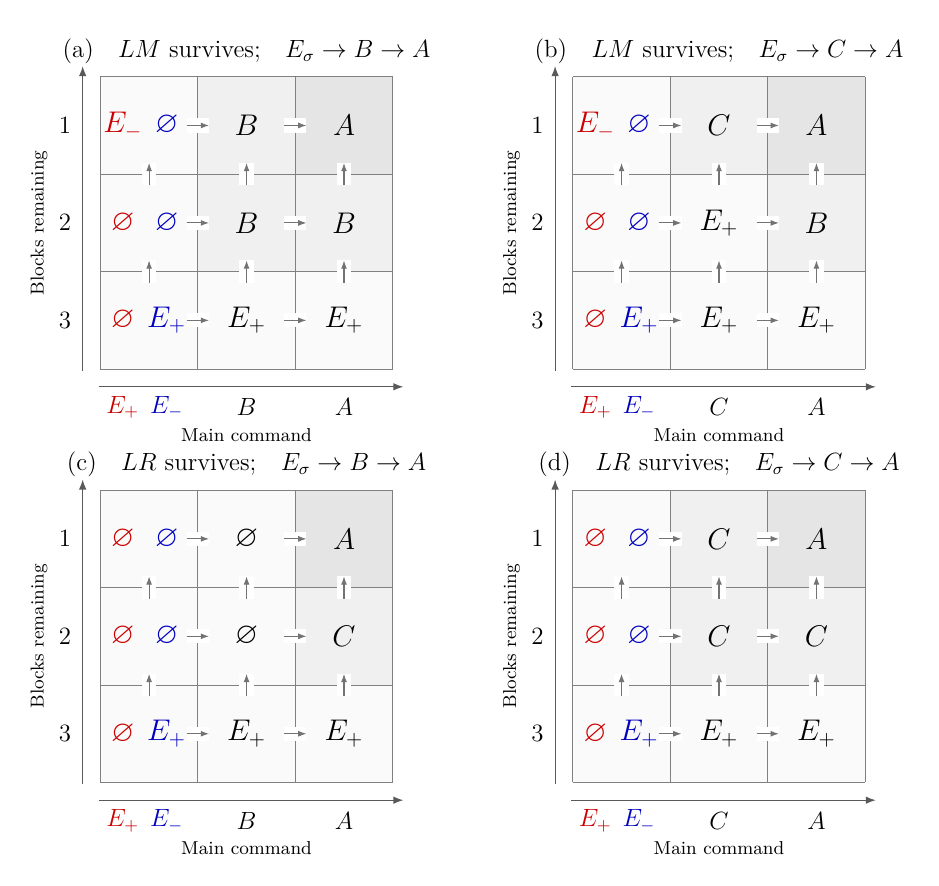}
    \caption{Command inheritance patterns for block L.}
    \label{fig:L}
\end{figure}

\begin{figure}[!htbp]
    \centering
    \includegraphics[width=0.72\linewidth]{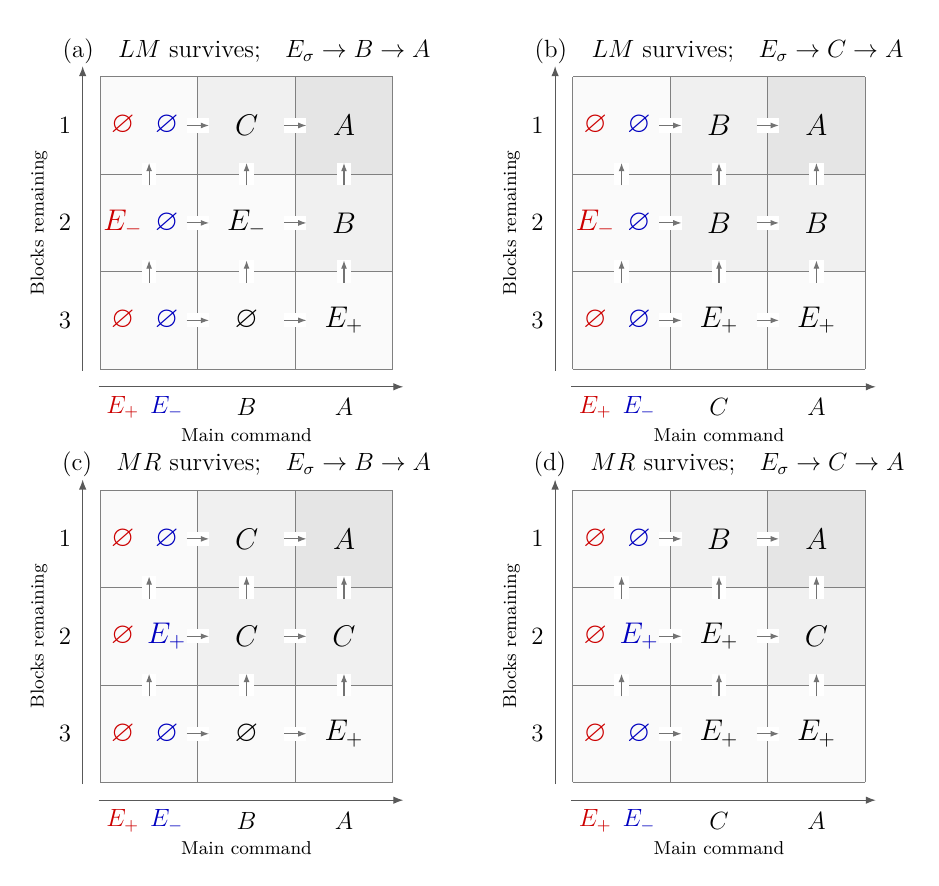}
    \caption{Command inheritance patterns for block M.}
    \label{fig:M}
\end{figure}

\begin{figure}[!htbp]
    \centering
    \includegraphics[width=0.72\linewidth]{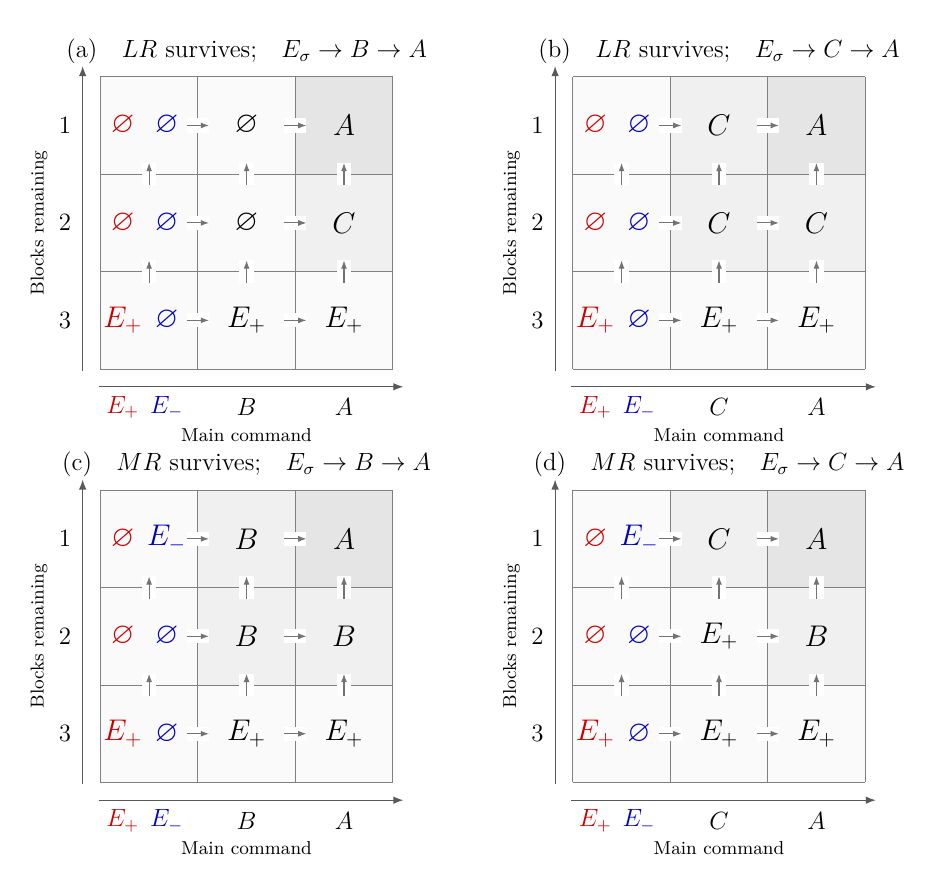}
    \caption{Command inheritance patterns for block R.}
    \label{fig:R}
\end{figure}
\FloatBarrier

\section{The recursive counting argument}\label{sec:recursion}
We now begin the proof of Theorem \ref{thm:main}. We first show how to use Lemma~\ref{lem:closure} to count the
peeling sequences recursively. The only part left from the proof for the next section
will be an estimate for the possible simplified peeling sequences.

Fix $\sigma,D,p,q$, and consider a peeling sequence counted by
$F_n^{\sigma,D}(p,q)$. At each step, write down $L,M$ or $R$ according
to which block the removed point belongs to. In this way we obtain a
sequence $\pi^*\in\{L,M,R\}^n$ with exactly $n_i$ occurrences of the symbol
$i\in\{L,M,R\}$. Following Dumitrescu and T\'oth~\cite{DT25}, we call $\pi^*$ the
\emph{simplified peeling sequence}. We call a simplified peeling
sequence \emph{valid} if it arises from at least one peeling sequence
counted by $F_n^{\sigma,D}(p,q)$.

The main command string and $\pi^*$ together determine the commands inherited
by each block. Indeed, the prefix of $\pi^*$ tells us which blocks are still
alive, and then Lemma~\ref{lem:router} gives the command inherited at the
next step. Conversely, if we choose a peeling sequence inside each block which
follows these inherited commands, we can combine the three peeling sequences defined on the blocks
according to $\pi^*$. The resulting sequence is a peeling sequence of the whole set following the fixed main commands, since the
rules in Lemma~\ref{lem:router} are exact.

Let $f$ be the first block emptied, and let $j,k$ be the other two
blocks, listed in the order $L,M,R$. By the last statement of Lemma~\ref{lem:closure}, the block $f$
inherits only $E_+$ commands, independent of the main command string.  For a fixed simplified
peeling sequence $\pi^*$, let $F_f(\pi^*)$, 
$F_j(\pi^*)$ and $F_k(\pi^*)$ denote the numbers of possible peeling sequences inside the
given block, with the inherited command structure determined by $\pi^*$. Block $f$ inherits the command string only containing $E_+$, so we can already tell $F_f(\pi^*)=1$ for any simplified peeling sequence. The following recursive equation holds by the bijection argument from the previous paragraph:
\begin{equation}\label{eq:simplified-sequences}
 F_n^{\sigma,D}(p,q)
 =\sum_{\pi^*\text{ valid}}F_f(\pi^*)F_j(\pi^*)F_k(\pi^*)=\sum_{\pi^*\text{ valid}}F_j(\pi^*)F_k(\pi^*).
\end{equation}
The blocks $f,j,k$, as well as the parameters in $F_j(\pi^*)$ and $F_k(\pi^*)$,
may depend on $\pi^*$.

We now define the products which will be used to upper bound the number of peeling sequences
inside the two remaining blocks following the inherited command from $\pi^*$. Set
\[
 t=\frac{1+\sqrt{17}}2=\beta-4,
 \qquad a=\frac{5\beta}{4},
\]
and define
\begin{equation}\label{eq:budgets}
 U_n^+(p,q)=t^p\beta^q a^{n-p-q},
 \qquad
 U_n^-(p,q)=4^p\beta^{n-p}.
\end{equation}
Each command contributes one factor to these products, according to the
following table:
\begin{equation}\label{eq:factors}
\begin{array}{c|ccc}
 \text{sign of the command string}&E_\sigma&D&A\\\hline
 +&t&\beta&a\\
 -&4&\beta&\beta.
\end{array}
\end{equation}
Thus a command string with positive endpoint sign, containing $p$ endpoint
commands and $q$ directional commands, receives the product
$t^p\beta^q a^{n-p-q}$. When we consider the first block emptied, we use
its exact number of sequences $1$ instead of the numbers in the table. If one of the
other blocks inherits no endpoint commands, we assign sign $+$ to it.
If it inherits no directional commands, the choice between $B$ and $C$
is arbitrary, since the proposed bound factors are the same for both.

Let
\[
 K(n)=48^{\,2^{\lceil\log_3 n\rceil}-1}.
\]
We will prove simultaneously, for both signs and both directions, the following:

\begin{lemma}\label{eq:induction} For any positive integer $n$, and command string defined by $\sigma,D,p,q$, the following holds:
 \[F_n^{\sigma,D}(p,q)\le K(n)U_n^\sigma(p,q).\]
\end{lemma}
Suppose for the moment that this has already been proved for smaller
sets, and put $m=\lceil n/3\rceil$. Each block has at most $m$ points.
Since $K$ is nondecreasing, for a fixed valid simplified peeling
sequence $\pi^*$ the induction gives
\[
 F_j(\pi^*)F_k(\pi^*)
 \le K(m)^2U_j(\pi^*)U_k(\pi^*),
\]
where $U_j(\pi^*)$ and $U_k(\pi^*)$ mean the expressions in
\eqref{eq:budgets}, with the parameters inherited by the corresponding
blocks. Substituting this into \eqref{eq:simplified-sequences}, we get
\begin{equation}\label{eq:induction-reduction}
 F_n^{\sigma,D}(p,q)
 \le K(m)^2\sum_{\pi^*\text{ valid}}U_j(\pi^*)U_k(\pi^*).
\end{equation}
Thus, to prove Lemma \ref{eq:induction}, it is enough to show that
\begin{equation}\label{eq:budgetsum}
 \boxed{
 \sum_{\pi^*\text{ valid}}U_j(\pi^*)U_k(\pi^*)
 \le48U_n^\sigma(p,q).}
\end{equation}
Figure~\ref{fig:proof-architecture} summarizes the argument. We prove
\eqref{eq:budgetsum} in Section~\ref{sec:weights}.

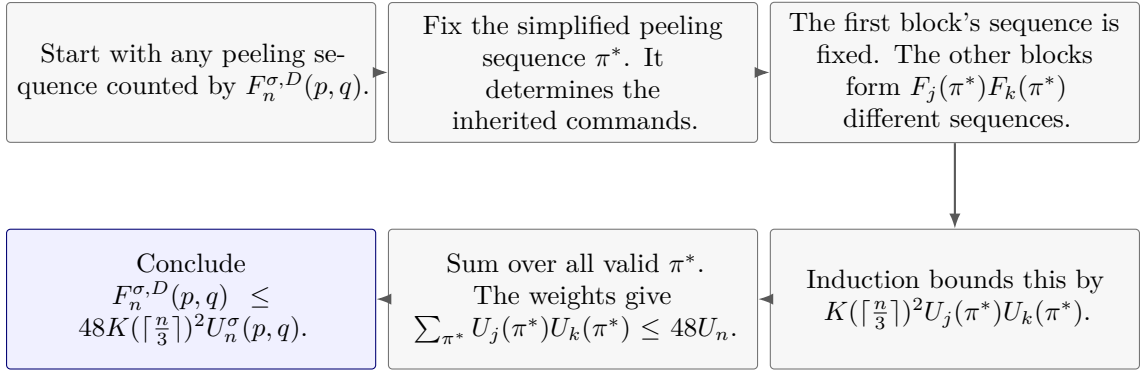
\begin{figure}[htbp]
\centering
\begin{tikzpicture}[
  box/.style={draw=black!55,fill=black!3,rounded corners=1.5pt,
    text width=4.6cm,minimum height=1.85cm,align=center,
    inner sep=4pt,font=\small},
  last/.style={box,draw=blue!45!black,fill=blue!6},
  arrow/.style={-{Latex[length=2mm]},thick,draw=black!65}]
  \node[box] (main) at (0,0)
    {Start with any peeling sequence counted by $F_n^{\sigma,D}(p,q)$.};
  \node[box] (word) at (5.05,0)
    {Fix the simplified peeling\\ sequence $\pi^*$. It determines the\\ inherited commands.};
  \node[box] (blocks) at (10.10,0)
    {The first block's sequence is fixed. The other blocks form $F_j(\pi^*)F_k(\pi^*)$ different sequences.};
  \node[box] (induction) at (10.10,-3.00)
    {Induction bounds this by\\ $K(\lceil\frac{n}{3}\rceil)^2U_j(\pi^*)U_k(\pi^*)$.};
  \node[box] (sum) at (5.05,-3)
    {Sum over all valid $\pi^*$.\\ The weights give\\ $\sum_{\pi^*}U_j(\pi^*)U_k(\pi^*)\le48U_n$.};
  \node[last] (conclusion) at (0,-3.00)
    {Conclude\\ $F_n^{\sigma,D}(p,q)\le48K(\lceil\frac{n}{3}\rceil)^2U_n^\sigma(p,q)$.};
  \draw[arrow] (main)--(word);
  \draw[arrow] (word)--(blocks);
  \draw[arrow] (blocks)--(induction);
  \draw[arrow] (induction)--(sum);
  \draw[arrow] (sum)--(conclusion);
\end{tikzpicture}
\caption{The structure of the recursive argument. The blocks $j,k$ and
their inherited parameters depend on the simplified peeling sequence
$\pi^*$.}
\label{fig:proof-architecture}
\end{figure}

We next check that \eqref{eq:budgetsum} is exactly the estimate needed
to complete the induction, and the base cases hold. For $n=1$, there is only $1$ sequence, while
$K(1)=1$ and every number in \eqref{eq:factors} is greater than $1$.
For $n=2$, there are at most two peeling sequences, while $K(2)=48$
and $U_2^\sigma(p,q)\ge1$. Thus Lemma~\ref{eq:induction} holds in both
base cases.
For $n\ge3$, equations \eqref{eq:induction-reduction} and
\eqref{eq:budgetsum} give
\[
 F_n^{\sigma,D}(p,q)
 \le48K(m)^2U_n^\sigma(p,q).
\]

If $h=\lceil\log_3n\rceil$, then
$\lceil\log_3m\rceil=h-1$. Consequently,
\[
 48K(m)^2
 =48^{\,1+2(2^{h-1}-1)}
 =48^{\,2^h-1}
 =K(n),
\]
which proves Lemma \ref{eq:induction}.

To prove Theorem \ref{thm:main}, take $p=q=0$ and use the negative sign. Then
$F_n^{-,D}(0,0)=g(S_n)$ and $U_n^-(0,0)=\beta^n$, giving
\[
 g(S_n)\le K(n)\beta^n.
\]
Since
$2^{\lceil\log_3n\rceil}\le2n^{\log_3 2}$ and
$\log_3 2<1$, we have $K(n)^{1/n}\to1$. Hence
\[
 \limsup_{n\to\infty}g(S_n)^{1/n}\le\beta.
\]
As $\beta<6.562<6.57$, the subexponential factor $K(n)$ can be
absorbed into $6.57^n$ for all sufficiently large $n$. The finitely
many smaller values of $n$ are absorbed into the constant in the
$O$-notation. Thus, once \eqref{eq:budgetsum} is proved, Theorem \ref{thm:main}
follows.

\section{The weight argument}\label{sec:weights}\label{sec:counting}
We now prove \eqref{eq:budgetsum}. Read a simplified peeling sequence
from left to right. After any fixed prefix, the next element records which
block supplies the next point. We assign weights so that the sum of weights
corresponding to all possible next entries is at most the weight
of the current main command.

We first split the valid simplified peeling sequences into classes $(f,\sigma_j,\sigma_k)$. A class specifies
the first block $f$ emptied and the endpoint signs inherited by the
other two blocks. There are three choices for $f$ and two choices for
each of the signs, so there are at most
\[
 3\cdot2^2=12
\]
classes. Some of these classes might be empty. As before, a block which
inherits no endpoint commands is assigned sign $+$. With this
notation, every simplified peeling sequence belongs to exactly one of the
classes.

Fix one class. We give the three blocks positive weights
$x_L,x_M,x_R$. If the next point is removed from block $i$, we multiply
the factor inherited by this block by $x_i$. We can do this because every simplified peeling sequence selects block $i$ exactly $n_i$
times. Therefore the product of all extra weights along any such
sequence is always
\[
 x_L^{n_L}x_M^{n_M}x_R^{n_R},
\]
independently of the order in which the blocks occur. Thus this extra
product depends only on the block sizes, and not on the sequence itself.
We can therefore use the weights at every step and divide out the same
product at the end.

We use the following weights based on $\sigma$ and the first disappearing block:
\begin{equation}\label{eq:weights}
\begin{array}{c|ccc}
\text{case}&x_L&x_M&x_R\\\hline
 \sigma=+,\ f=R&4/\beta&t/4&t\\
 \text{all other cases}&\multicolumn{3}{c}
 {x_f=4,\quad x_i=1/2\quad(i\ne f).}
\end{array}
\end{equation}
The first row is needed in the exceptional case when the main sign is
$+$ and $R$ disappears first. A positive main endpoint stage visits
the blocks in the order $R,M,L$, and the later
blocks may inherit negative endpoint commands, whose factor is $4$, even
though the factor of the main command is only $t$. The asymmetric
weights in the first row balance these three possibilities. In every
other case, the simpler second row is sufficient.

In each row, the weights satisfy
\begin{equation}\label{eq:product}
 x_Lx_Mx_R=1,
 \qquad \frac12\le x_i\le4.
\end{equation}
For the first row, the product is $t^2/\beta=1$; the second row has
product $4\cdot(1/2)^2=1$. The bounds in the first row follow from
$5/2<t<8/3$. We will use the elementary relations
\begin{equation}\label{eq:identities}
 t^2=\beta=t+4,
 \qquad
 \frac4\beta+\frac t4=\frac54,
 \qquad
 \frac{16}\beta\le t,
 \qquad
 \frac{9t}{4}\le\beta.
\end{equation}
Here $5/2<t<8/3$. The second identity follows after multiplying by
$4\beta$ and using $t^3=5t+4$. The third follows from
$t^3=5t+4>16$, and the fourth from $t>9/4$ and $\beta=t^2$. As a reminder, 
\[
 t=\frac{1+\sqrt{17}}2=\beta-4,
 \qquad a=\frac{5\beta}{4}.
\]

\begin{lemma}\label{lem:weights}
Fix a class and a prefix which can be completed to a simplified peeling
sequence in this class. For every possible next block $i$, let $c_i$ be
the number
assigned in \eqref{eq:factors} to the command inherited by $i$. For the
first block emptied, we use $c_f=1$. Let $c$ be the number assigned to
the current main command. Then
\[
 \sum_i x_i c_i\le c.
\]
\end{lemma}

\begin{proof}
The endpoint signs of the two blocks other than $f$ are fixed in the class.
Therefore, if a possible next deletion would give one of them the wrong
endpoint sign, it does not occur among the choices in the sum. In
particular, whenever one of these two blocks inherits $E_+$, its factor
is $t$.

First consider the second row of \eqref{eq:weights}. While all three
blocks are alive and the main endpoint stage has already finished, the
first block contributes at most $4\cdot1$, and each of the other two
blocks contributes at most $(1/2)t$. Hence the total is at most
\[
 4+\frac t2+\frac t2=4+t=\beta.
\]
This is at most the main factor for $B,C$ or $A$. After $f$ disappears,
the two remaining blocks both have weight $1/2$. While both remain,
each inherited command is either an endpoint command or a directional
command, with factor at most $\beta$. Their total contribution is therefore at most
$(1/2)\beta+(1/2)\beta=\beta$. When only one block remains, its contribution is at most
$a/2\le\beta$ under command $A$, and at most $\beta/2$ under $B$ or
$C$.

It remains to check the main endpoint stage. If its sign is negative,
only one block can supply the next point. Its contribution is
$4\cdot1$ if this is the first block emptied, and at most
$(1/2)\cdot4$ otherwise. Both are at most the main factor $4$. If the
main sign is positive, then in the second row of \eqref{eq:weights} we
have $f\ne R$. Thus the endpoint stage must finish before $R$
disappears. During this whole stage only $R$ is selected, and its
contribution is $t/2\le t$.

Now consider the first row of \eqref{eq:weights}. Here the main sign is
positive and $R$ disappears first. During the main endpoint stage, the
blocks are visited in the order $R,M,L$. The corresponding
contributions are at most
\[
 x_R\cdot1=t,
 \qquad x_M\cdot4=t,
 \qquad x_L\cdot4=\frac{16}{\beta}\le t.
\]
Thus every step is bounded by the main factor $t$.

Outside the endpoint stage, while all three blocks are alive, the block
$R$ contributes at most $t\cdot1$, and the other blocks together
contribute at most $t(x_L+x_M)$. Hence the total is at most
\[
 t+t(x_L+x_M)=\frac{9t}{4}\le\beta.
\]
After $R$ disappears, the remaining pair is $LM$. Reading the row $LM$
of the table in Lemma~\ref{lem:router}, we get
\[
\begin{array}{c|c|c}
 \text{main command}&\text{upper bound on the sum}&
     \text{main factor}\\\hline
 A&\beta(x_L+x_M)=5\beta/4&a\\
 B&\beta x_L+4x_M=4+t&\beta\\
 C&t x_L+\beta x_M=9t/4&\beta.
\end{array}
\]
Every entry in the middle column is at most the corresponding main
factor, by \eqref{eq:identities}. Finally, suppose only one of $L,M$
remains. Outside the endpoint stage, its weight is at most $1$, and its
local factor is at most the corresponding main factor. The endpoint
cases were already checked above. This proves the lemma.
\end{proof}

Lemma~\ref{lem:weights} concerns only one deletion. We now apply it
repeatedly to obtain a bound for complete simplified peeling sequences.
The following argument is the weighted prefix-tree form of the
Kraft--McMillan inequality; see \cite[Lemma~1.2]{HJM26}. For its
coding-theoretic background, see \cite{Kraft49,McMillan56} and
\cite[Chapter~5]{CT06}. We include the argument to keep the proof
self-contained.
Fix one of the twelve classes, and consider the tree formed by the
prefixes of the sequences in this class. We only keep prefixes which
can be completed to a simplified peeling sequence in the class. An edge labelled $i$
means that the next point is removed from block $i$. Using the notation
of Lemma~\ref{lem:weights}, give this edge the number
\begin{equation}\label{eq:edge}
 \rho_i=\frac{x_i c_i}{c}.
\end{equation}
The inherited command is determined by the prefix. Its factor is also
determined, because the endpoint signs are fixed within the class.

Give the root of the tree value $1$. The value of every other node is
the product of the numbers on the path from the root to that node. By
Lemma~\ref{lem:weights}, the sum of the outgoing edge numbers is at most
$1$. Therefore the total value cannot increase when we pass from one
level of the tree to the next. It follows that the sum of the values of
all leaves is at most $1$.

For a complete simplified peeling sequence $\pi^*$, the product of the edge
numbers on its path is
\[
 \left(\prod_i x_i^{n_i}\right)
 \frac{U_j(\pi^*)U_k(\pi^*)}{U_n^\sigma(p,q)}.
\]
Indeed, the numerator contains the factors inherited by the two blocks
other than $f$, while every deletion in $f$ has factor $1$. The
denominator contains exactly the factors of the main command string.
The extra weights give $x_i$ once every time block $i$ occurs.

The product $\prod_i x_i^{n_i}$ is the same for every simplified peeling
sequence in the class. Summing over the leaves therefore gives
\begin{equation}\label{eq:sum}
 \sum_{\pi^*\text{ in the class}}U_j(\pi^*)U_k(\pi^*)
 \le
 \frac{U_n^\sigma(p,q)}{\prod_i x_i^{n_i}}.
\end{equation}
Write $n=3s+r$, where $0\le r\le2$. Every block has size $s$ or
$s+1$, and exactly $r$ blocks have the larger size. Using
\eqref{eq:product},
\[
 \prod_i x_i^{n_i}
 =(x_Lx_Mx_R)^s
   \prod_{i:\,n_i=s+1}x_i
 \ge\left(\frac12\right)^r
 \ge\frac14.
\]
Thus the contribution of one class is at most
$4U_n^\sigma(p,q)$. There are at most twelve classes, so
\[
 \sum_{\pi^*\text{ valid}}U_j(\pi^*)U_k(\pi^*)
 \le12\cdot4U_n^\sigma(p,q)
 =48U_n^\sigma(p,q).
\]
This proves \eqref{eq:budgetsum}, and completes the proof of
Theorem~\ref{thm:main}.

\section{A lower bound for the 3-block construction}\label{sec:lower}

In this section we prove Theorem~\ref{thm:lowerbound}.
Throughout the section we assume $n=3^k$ for some integer $k\geq2$. We use the command notations as before, together with the inheritance rules from Lemma~\ref{lem:router}.
We distinguish three different removal types.
Type $X$ means that for each of the $n$ steps, we remove a vertex according to command $C$.
Type $Y$ means that for each of the first $\frac{4}{9}n$ steps we remove according to command $E_+$, and for the remaining vertices remove according to command $C$.
Type $Z$ means that for the first $\frac{1}{3}n$ steps we remove a vertex according to command $E_-$, and then remove according to command $B$ for the rest of the steps.
Formally:
\[
X=C^n,\qquad
Y=E_+^{\frac{4}{9}n}C^{\frac{5}{9}n},\qquad
Z=E_-^{\frac{1}{3}n} B^{\frac{2}{3}n}.
\]
Let $X(n),Y(n),Z(n)$ denote the number of peeling sequences of $S_n$ following removal type $X,Y,Z$ respectively.
We prove the following:

\begin{theorem}\label{thm:lowertypes}
Let $n=3^k$ for a sufficiently large positive integer $k$.
Then,
\[
X(n)\geq 5.22^n,\qquad
Y(n)\geq 3.10^n,\qquad
Z(n)\geq 4.57^n.
\]
\end{theorem}

\begin{proof}
We bound $X(n),Y(n),Z(n)$ from below as a function of $X(\frac{n}{3}),Y(\frac{n}{3}),Z(\frac{n}{3})$.
\begin{itemize}
\item $X(n)$: We only calculate a sufficiently large subset of such peeling sequences.
Consider only those, where in the first $\frac{17}{27}n$ steps we remove $\frac{4}{27}n$ vertices from block $L$, $\frac{1}{3}n$ vertices from block $M$ and $\frac{4}{27}n$ vertices from block $R$.
Even if $M$ disappears earlier, we just remove the $E_+$ point of each block in these steps, which can only decrease the number of sequences we count.
On the remaining $L,R$ blocks, the inherited commands from $C$ are $(C,C)$.
So we take all the possible simplified peeling sequences on these $\frac{10}{27}n$ points, and remove from each block according to command $C$, even if the other block already disappeared.
Now notice that this removal type forces an $E_+$ command on the whole block $M$, forces $E_+$ on the first $\frac{4}{27}n$ points of $L$ and $R$, and forces command $C$ on the last $\frac{5}{27}n$ points of $L$ and $R$, so we enforce removal type $Y$ on $L$ and $R$.
Using that any simplified peeling sequence combined with any peeling sequences on the blocks obeying the inherited commands gives a unique peeling sequence, we can give a lower bound on $X(n)$ now.
\[
X(n)\geq
\binom{\frac{17}{27}n}{\frac{4}{27}n,\frac{9}{27}n,\frac{4}{27}n}
\cdot\binom{\frac{10}{27}n}{\frac{5}{27}n}
\cdot Y\!\left(\frac{n}{3}\right)^2.
\]
\item $Y(n)$: We are again looking for a lower bound only.
For the first $\frac49 n$ steps, $E_+$ forces the removal order.
After that, we apply command $C$ on the remaining set, which is the entire $L$ block, and the rest of the $M$ block after removing $\frac19 n$ vertices from it via the local $E_-$ command.
The inherited commands from $C$ to $L,M$ are $(E_+,B)$.
We calculate the sequences where the next $\frac{10}{27}n$ steps remove $M$, and $\frac4{27}n$ vertices of $L$ under inherited command $E_+$. Even if $M$ disappears earlier, we remove from $L$ according to command $E_+$ here.
After that, $\frac5{27}n$ vertices of $L$ remain under inherited command $C$.
Observe, that this removal pattern enforces removal type $Z$ on $M$, and removal type $Y$ on $L$.
Hence:
\[
Y(n)\geq \binom{\frac{10}{27}n}{\frac4{27}n}\cdot Y\left(\frac n3\right)\cdot Z\left(\frac n3\right).
\]

\item $Z(n)$: The first $\frac n3$ commands are $E_-$, removing $L$.
The next commands are $B$, the inherited commands of $M,R$ are $(C,B)$.
We consider any possible simplified peeling sequence on the remaining $\frac{2n}{3}$ points.
On block $M$, this gives removal type $X$.
Since we look for a lower bound, on block $R$ we can take command $E_-$ for the first $\frac n9$ steps instead of $B$, to get removal type $Z$ on $R$.
Hence we get the lower bound:
\[
Z(n)\geq \binom{\frac23 n}{\frac13 n}\cdot X\left(\frac n3\right)\cdot Z\left(\frac n3\right).
\]
\end{itemize}

We use the standard entropy estimate
\[
\begin{gathered}
\log\binom{N}{Np_1,\ldots,Np_r}
=NH(p_1,\ldots,p_r)+O(\log N),\\
H(p_1,\ldots,p_r)=-\sum_{i=1}^r p_i\log p_i,
\end{gathered}
\]
for fixed positive $p_i$ with $\sum_i p_i=1$ and integral $Np_i$;
see~\cite{Robbins1955,Csiszar1998,OikonomouGrunwald2016}.
All logarithms are natural. Put
\[
\begin{aligned}
a&=\frac{17}{27}H\!\left(\frac4{17},\frac9{17},\frac4{17}\right)
  +\frac{10}{27}\log2,\\
b&=\frac{10}{27}H\!\left(\frac25,\frac35\right),
\qquad c=\frac23\log2.
\end{aligned}
\]
Let $\xi,\eta,\zeta$ be the lower limits of
$\log X(n)/n$, $\log Y(n)/n$, and $\log Z(n)/n$, respectively,
as $n$ tends to infinity through powers of three.
Taking logarithms in the three recursive inequalities, dividing by $n$,
and taking lower limits gives
\[
\xi\geq a+\frac23\eta,\qquad
\eta\geq b+\frac{\eta+\zeta}{3},\qquad
\zeta\geq c+\frac{\xi+\zeta}{3}.
\]
If any of these lower limits is infinite, the inequalities force all
three to be infinite and the conclusion is immediate. Otherwise,
substituting these inequalities into one another yields
\[
\begin{aligned}
\xi&\geq\frac{6a+6b+3c}{5}>\log5.22,\\
\eta&\geq\frac{3a+18b+9c}{10}>\log3.10,\\
\zeta&\geq\frac{3a+3b+9c}{5}>\log4.57.
\end{aligned}
\]
The strict inequalities imply the three claimed bounds for all
sufficiently large powers of three.
\end{proof}

\begin{proof}[Proof of Theorem \ref{thm:lowerbound}:] Running command $A$ instead of $C$ can only increase the number of peeling orders, hence $g(S_n)\geq X(n)$ when $n$ is a power of three.
This proves Theorem~\ref{thm:lowerbound}. \end{proof}

\section{Concluding Remarks}

 The upper-bound method used in this paper can be refined by keeping
track of the directions of the $n/9,n/27,\ldots$ subblocks and choosing
weights for the resulting command strings. Such refinements require more
involved calculations; we do not state any additional numerical bound here.

 The most interesting question still remains open. Can the lower bound be improved from the almost trivial $g(n)\geq c\cdot3^n$ where $c$ is some positive constant? A polynomial factor would already be interesting.

\clearpage
\appendix
\section{Proof of Lemma \ref{lem:router}}\label{sec:coordinates}

\subsection{The recursive coordinates}

We construct $S_n$ by induction on $n$, maintaining distinct first
coordinates and general position. Let $S_1=\{(0,0)\}$. For $n\ge2$,
use blocks of sizes
\[
 n_L=\lfloor n/3\rfloor,\qquad
 n_M=\lfloor(n+1)/3\rfloor,\qquad
 n_R=\lfloor(n+2)/3\rfloor,
\] respectively for the left, middle and right blocks, 
and omit an empty block.

Before embedding a recursive block, translate it and multiply
both coordinates by a positive number so that it lies in $[-1,1]^2$. Also, in this state we refer to the block as directed, and its direction is towards the positive direction on the $x$-axis.
These operations preserve every deletion command in Lemma \ref{lem:router}.

Fix
\[
 v_L=(-2,1),\qquad v_M=(-1,-1),\qquad v_R=(3,0),
 \qquad (\chi_L,\chi_M,\chi_R)=(-1,1,1).
\]
Let $J(x,y)=(-y,x)$ and fix $\delta=1/10$. Embed block $i$ by
\begin{equation}\label{geom:embedding}
 \Phi_i(x,y)=(1+\delta x)v_i+\varepsilon\chi_i yJv_i.
\end{equation}

The parameter $\varepsilon>0$ is chosen sufficiently small for the three
finite blocks being embedded. It may be different at different nodes.
The union of these embedded blocks is $S_n$.

The role of the two terms is concrete: The first term moves our blocks onto a small segment on its corresponding branch, and the second term adds back the affine geometry of the blocks, while simultaneously making them sufficiently flattened. The $x$-coordinate of the three blocks in this construction increases in this order: $L,M,R$, and no two points share the same first coordinate.
Their local first-coordinate orders are reversed in $L,M$ and preserved in $R$, since all of the blocks point outwards from the origin. In addition to that, on block $L$ we did a reflection across its own $x$-axis before embedding, which will give us crucial consequences. These facts still hold when $\varepsilon$ is sufficiently small. Also, general position is preserved, since each block will be sufficiently flat such that $3$ points from more than one block cannot be collinear, and any $3$ points within the same block stay non-collinear due to the flattening being an affine transformation. For visual understanding, we refer to Figure \ref{geom:figure}.

\begin{figure}[!htbp]
\centering
\includegraphics[width=0.75\linewidth]{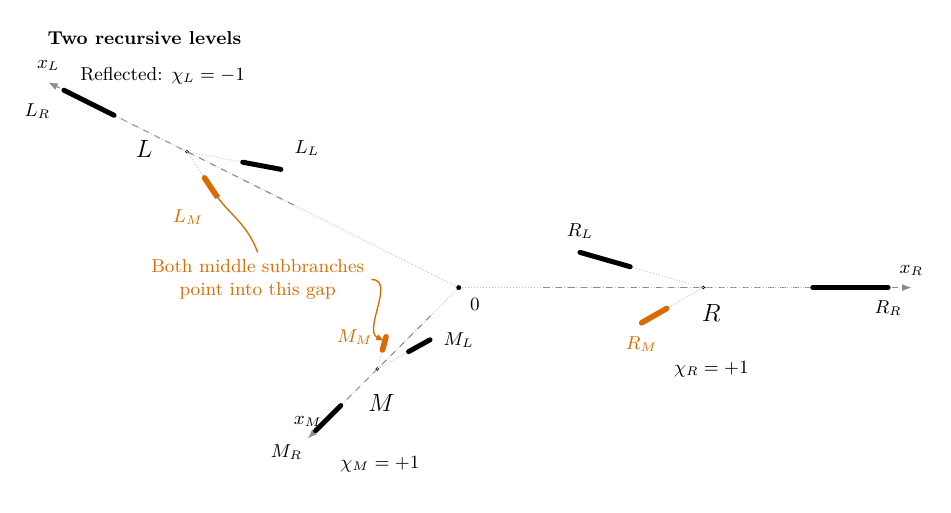}
\caption{The building of $S_n$ from the three blocks. All of them are placed pointing outwards from the origin, with the key difference of $L$ being reflected across its own $x$-axis. This is going to have important consequences in the removal command inheritance.}
\label{geom:figure}
\end{figure}

\FloatBarrier
\subsection{Verification of command inheritance to the blocks}
\begin{proof}
 The commands inherited by the blocks can be explained by a careful analysis of Figure \ref{geom:figure}.

 When all three blocks are still alive, commands $A$ and $C$ can remove each block's outermost point, so it corresponds to an $E_+$ on each of them. Command $B$ cannot touch the middle block since it is not visible from above, and acts as an $E_+$ on the other blocks.

 Once the first block disappears, there are three possible combinations of the remaining blocks: $LM$, $LR$ and $MR$. Filling in the table of the lemma is just a case analysis, which we will illustrate by pictures in Figure \ref{geom:figure9}.

 \[
\begin{array}{c|ccc}
\text{blocks}&A&B&C\\\hline
LM&(B,B)&(B,E_-)&(E_+,B)\\
LR&(C,C)&(E_+,E_+)&(C,C)\\
MR&(C,B)&(C,B)&(E_+,E_+).
\end{array}
\]

We describe an example case, all entries of the table are determined similarly. Let us consider row $LM$ and column $B$, which is depicted in Figure \ref{geom:figureLMB}. This means, blocks $L$ and $M$ remain in the construction, and we are only allowed to peel points from the convex hull that are visible from above. The analysis of this table entry is explained in the caption of Figure \ref{geom:figureLMB}.

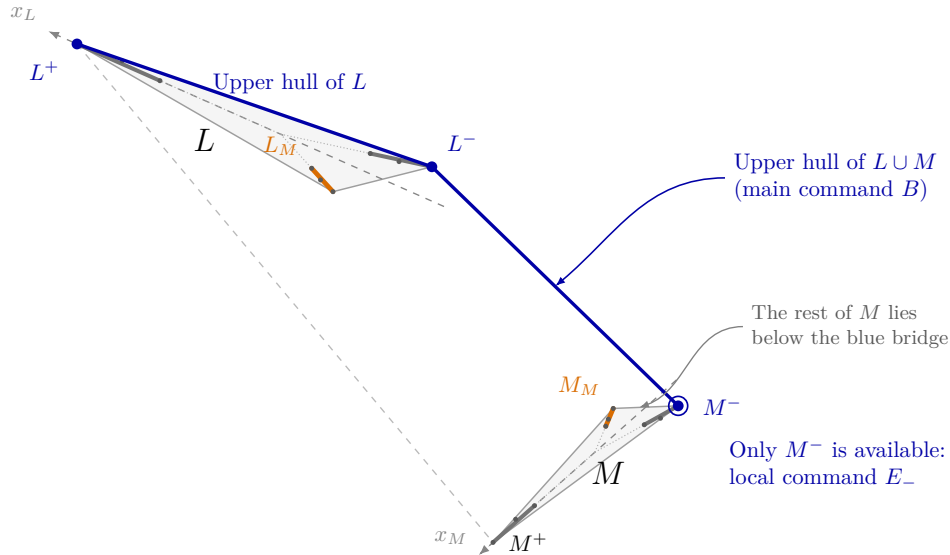
\begin{figure}[!htbp]
\centering
\begingroup
\setlength{\fboxsep}{5mm}
\resizebox{0.85\linewidth}{!}{\colorbox{white}{\begin{tikzpicture}[x=1.40cm,y=1.24cm,>=Latex,
 every node/.style={font=\small},
 guide/.style={black!32,densely dotted,line width=.5pt},
 axis/.style={black!48,dashed,line width=.6pt,-{Latex[length=1.7mm]}},
 subbranch/.style={black!55,line width=1.8pt,line cap=round},
 middle/.style={orange!85!black,line width=2.1pt,line cap=round},
 exposed/.style={blue!65!black,line width=1.5pt,line join=round}]
\path[use as bounding box] (-8.15,-3.15) rectangle (2.5,4.15);

\foreach \name/\cx/\cy/\vx/\vy/\sgn in
 {L/-5/2.5/-2/1/-1,M/-1.5/-1.5/-1/-1/1}{
 \coordinate (\name-O) at (\cx,\cy);
 \coordinate (\name-axisstart) at ({\cx-2.6*.35*\vx},{\cy-2.6*.35*\vy});
 \coordinate (\name-axisend) at ({\cx+3.7*.35*\vx},{\cy+3.7*.35*\vy});
 \foreach \sub/\ux/\uy in {L/-2/1,M/-1/-1,R/3/0}{
  \pgfmathsetmacro{\dx}{.35*\ux*\vx-.16*\sgn*\uy*\vy}
  \pgfmathsetmacro{\dy}{.35*\ux*\vy+.16*\sgn*\uy*\vx}
  \coordinate (\name-\sub-inner) at ({\cx+.64*\dx},{\cy+.64*\dy});
  \coordinate (\name-\sub-mid) at ({\cx+.84*\dx+.009*\uy},{\cy+.84*\dy+.013*\ux});
  \coordinate (\name-\sub-outer) at ({\cx+1.08*\dx},{\cy+1.08*\dy});
 }
}

\filldraw[fill=black!4,draw=black!38,line width=.65pt]
 (L-R-outer) -- (L-L-outer) -- (L-M-outer) -- cycle;
\filldraw[fill=black!4,draw=black!38,line width=.65pt]
 (M-R-outer) -- (M-M-outer) -- (M-L-outer) -- cycle;
\draw[black!27,dashed,line width=.6pt] (L-R-outer) -- (M-R-outer);

\foreach \name in {L,M}{
 \draw[axis] (\name-axisstart) -- (\name-axisend)
   node[above left=1pt] {$x_{\name}$};
 \foreach \sub in {L,M,R}{
  \draw[guide] (\name-O) -- (\name-\sub-inner);
  \ifnum\pdfstrcmp{\sub}{M}=0
   \draw[middle] (\name-\sub-inner) -- (\name-\sub-outer);
  \else
   \draw[subbranch] (\name-\sub-inner) -- (\name-\sub-outer);
  \fi
  \foreach \pos in {inner,mid,outer}
   \fill[black!65] (\name-\sub-\pos) circle (1.1pt);
 }
}

\draw[exposed] (L-R-outer) -- (L-L-outer) -- (M-L-outer);
\foreach \pt in {L-R-outer,L-L-outer,M-L-outer}
 \fill[blue!65!black] (\pt) circle (2.4pt);
\draw[blue!65!black,line width=.75pt] (M-L-outer) circle (4.2pt);

\node[below left=4pt,text=blue!65!black] at (L-R-outer) {$L^+$};
\node[above right=3pt,text=blue!65!black] at (L-L-outer) {$L^-$};
\node[right=3pt] at (M-R-outer) {$M^+$};
\node[right=7pt,text=blue!65!black] at (M-L-outer) {$M^-$};

\node[font=\Large] at (-5.85,2.43) {$L$};
\node[font=\Large] at (-1.35,-1.76) {$M$};
\node[above=7pt,text=blue!65!black,align=center]
 at ($(L-R-outer)!.60!(L-L-outer)$)
 {Upper hull of $L$};

\node[above left=2pt,text=orange!85!black] at (L-M-inner) {$L_M$};
\node[above left=2pt,text=orange!85!black] at (M-M-outer) {$M_M$};

\node[anchor=west,text=blue!65!black,align=left] (parentupper)
 at (-.05,1.95) {Upper hull of $L\cup M$\\(main command $B$)};
\draw[blue!65!black,-{Latex[length=1.6mm]},line width=.7pt]
 (parentupper.west) to[out=180,in=40] ($(L-L-outer)!.50!(M-L-outer)$);

\node[anchor=west,align=left,font=\footnotesize,text=black!65] (hidden)
 at (.15,.08) {The rest of $M$ lies\\below the blue bridge};
\draw[black!50,-{Latex[length=1.5mm]},line width=.6pt]
 (hidden.west) to[out=180,in=20] ($(M-M-outer)!.40!(M-L-outer)$);
\node[anchor=west,text=blue!65!black,align=left]
 at (-.1,-1.65) {Only $M^-$ is available:\\local command $E_-$};

\end{tikzpicture}}}
\endgroup
\caption{Blocks $L,M$ remaining under command $B$. Only vertices on the blue upper hull of the whole set may be removed. This means all vertices on the upper hull of $L$ may be removed, so $L$ inherits command $B$. For $M$, only the rightmost vertex may be removed, which is at its negative direction, so $M$ inherits command $E_-$. So the corresponding table entry must be $(B,E_-)$.}
\label{geom:figureLMB}
\end{figure}

An exhaustive drawing of all $9$ table entries can be checked on Figure \ref{geom:figure9}. 

Once only one block remains, the statement of the lemma becomes obvious. For $L$ and $R$ the commands $A,B,C$ all inherit from the main set. For $M$, the commands $B,C$ get flipped. The inheritance of commands $E_+$ and $E_-$ is also straightforward if one just looks at the direction of these blocks on the drawing.

\begin{figure}[!p]
\centering
\includegraphics[width=\linewidth,height=0.86\textheight,keepaspectratio]{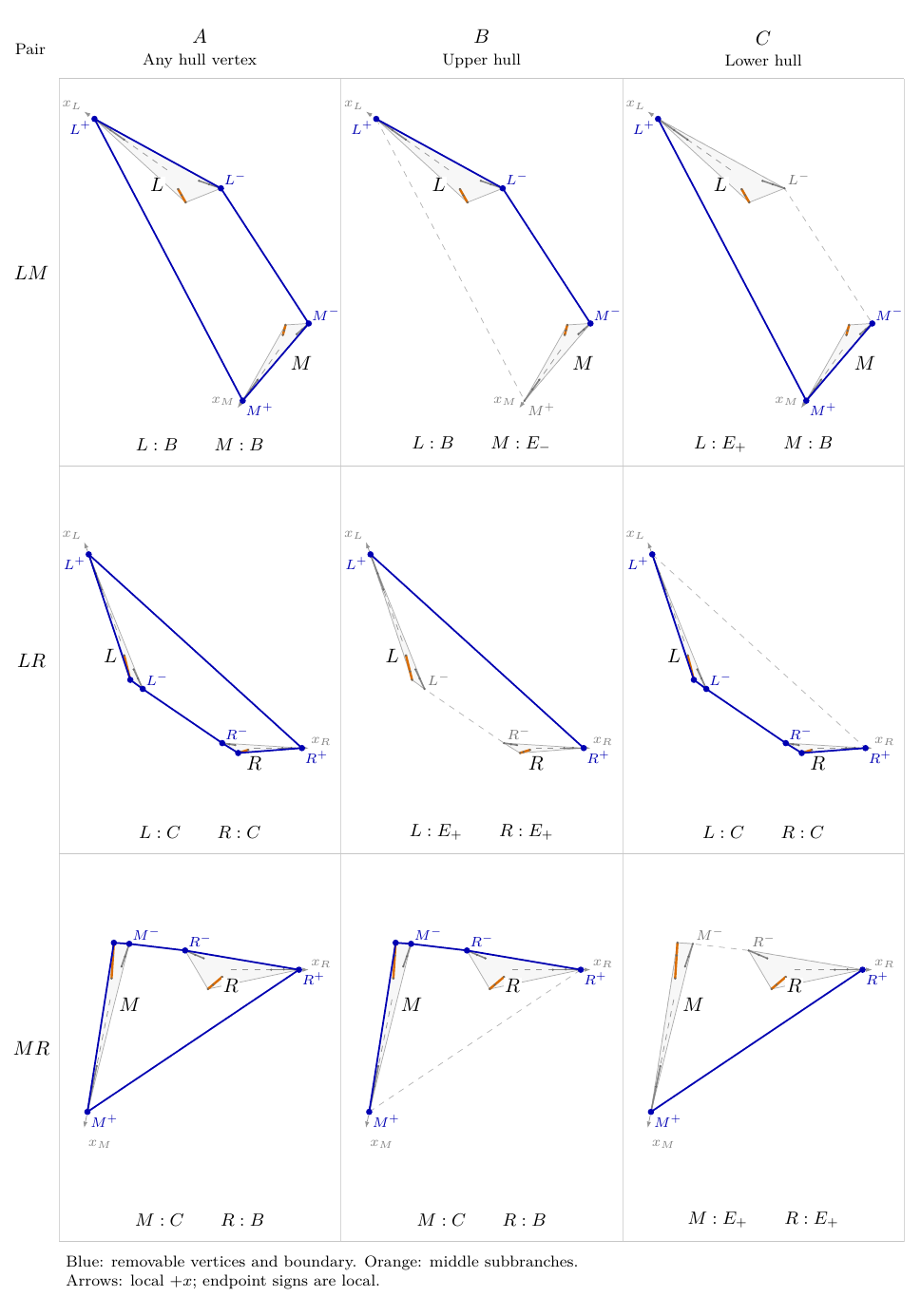}
\caption{The blue parts correspond to which parts of the drawing we can remove vertices from. Under each diagram, it is written what command the drawing corresponds to on the remaining blocks.}
\label{geom:figure9}
\end{figure}

Note that missing vertices from the blocks do not alter the commands they inherit, where each command is applied to the points which remain. Indeed, the points within each block have distinct local first coordinates. Since there are only finitely many pairs of points, we can choose $\epsilon$ sufficiently small that every segment joining two points within a block is arbitrarily close to parallel with its branch.

For each other block, its branch lies strictly on one side of the line containing our branch. Therefore, by choosing $\epsilon$ sufficiently small, all points of that other block lie on the same prescribed side of every line determined by two points of our block. These side relations are exactly those used in the geometric
arguments above, and deleting points cannot change them.

Hence, missing vertices do not change whether the block inherits an endpoint command, an upwards or downwards visibility command, or command $A$. This depends solely on the main command and the currently surviving blocks. This concludes the proof. \end{proof}

\section*{Statements and Declarations}

\paragraph{Competing interests.}
The author has no relevant financial or non-financial
interests to disclose.

\paragraph{Data availability.}
Data sharing is not applicable to this article as no
datasets were generated or analysed during the current study.

\paragraph{Funding.} 
Research was supported by the ERC Advanced Grant ``Geoscape''.

\FloatBarrier


\begin{thebibliography}{99}

\bibitem{Barnett76}
V. Barnett, The ordering of multivariate data,
\emph{J. Roy. Statist. Soc. Ser. A}
\textbf{139} (1976), no.~3, 318--354.
doi:\href{https://doi.org/10.2307/2344839}
{\nolinkurl{10.2307/2344839}}.

\bibitem{Chazelle85}
B. Chazelle, On the convex layers of a planar set,
\emph{IEEE Trans. Inform. Theory}
\textbf{31} (1985), no.~4, 509--517.
doi:\href{https://doi.org/10.1109/TIT.1985.1057060}
{\nolinkurl{10.1109/TIT.1985.1057060}}.

\bibitem{CT06}
T. M. Cover and J. A. Thomas,
\emph{Elements of Information Theory},
2nd ed., Wiley-Interscience, Hoboken, NJ, 2006.
doi:\href{https://doi.org/10.1002/047174882X}
{\nolinkurl{10.1002/047174882X}}.

\bibitem{Csiszar1998}
I. Csisz\'ar, The method of types,
\emph{IEEE Trans. Inform. Theory}
\textbf{44} (1998), no.~6, 2505--2523.
doi:\href{https://doi.org/10.1109/18.720546}
{\nolinkurl{10.1109/18.720546}}.

\bibitem{Dum22}
A. Dumitrescu, Peeling sequences,
\emph{Mathematics} \textbf{10} (2022), no.~22, article 4287.
doi:\href{https://doi.org/10.3390/math10224287}{\nolinkurl{10.3390/math10224287}}.

\bibitem{DT25}
A. Dumitrescu and G. T\'oth, Peeling sequences,
\emph{Discrete Comput. Geom.} \textbf{73} (2025), no.~3, 837--849.
doi:\href{https://doi.org/10.1007/s00454-023-00616-8}{\nolinkurl{10.1007/s00454-023-00616-8}}.

\bibitem{EW85}
H. Edelsbrunner and E. Welzl,
On the number of line separations of a finite set in the plane,
\emph{J. Combin. Theory Ser. A}
\textbf{38} (1985), no.~1, 15--29.
doi:\href{https://doi.org/10.1016/0097-3165(85)90017-2}
{\nolinkurl{10.1016/0097-3165(85)90017-2}}.

\bibitem{HL13}
S. Har-Peled and B. Lidick\'y, Peeling the grid,
\emph{SIAM J. Discrete Math.}
\textbf{27} (2013), no.~2, 650--655.
doi:\href{https://doi.org/10.1137/120892660}
{\nolinkurl{10.1137/120892660}}.

\bibitem{HJM26}
A. Hisatsuga, G. Johnston, and R. Miyazaki,
A base-$8$ upper bound for planar peeling sequences,
\href{https://arxiv.org/abs/2609.13122}{arXiv:2609.13122} (2026).

\bibitem{Kraft49}
L. G. Kraft,
\emph{A device for quantizing, grouping, and coding
amplitude-modulated pulses},
M.S. thesis, Department of Electrical Engineering,
Massachusetts Institute of Technology, 1949.
\url{https://dspace.mit.edu/handle/1721.1/12390}.

\bibitem{McMillan56}
B. McMillan,
Two inequalities implied by unique decipherability,
\emph{IRE Trans. Inform. Theory}
\textbf{2} (1956), no.~4, 115--116.
doi:\href{https://doi.org/10.1109/TIT.1956.1056818}
{\nolinkurl{10.1109/TIT.1956.1056818}}.

\bibitem{OikonomouGrunwald2016}
K. N. Oikonomou and P. D. Gr\"unwald,
Explicit bounds for entropy concentration under linear constraints,
\emph{IEEE Trans. Inform. Theory}
\textbf{62} (2016), no.~3, 1206--1230.
doi:\href{https://doi.org/10.1109/TIT.2015.2458951}
{\nolinkurl{10.1109/TIT.2015.2458951}}.

\bibitem{Robbins1955}
H. Robbins, A remark on Stirling's formula,
\emph{Amer. Math. Monthly}
\textbf{62} (1955), no.~1, 26--29.
doi:\href{https://doi.org/10.2307/2308012}
{\nolinkurl{10.2307/2308012}}.

\bibitem{SimonHigher26}
D. G. Simon,
The minimum number of peeling sequences of a point set,
\emph{Discrete Comput. Geom.}
\textbf{75} (2026), no.~4, 1122--1133.
doi:\href{https://doi.org/10.1007/s00454-024-00713-2}
{\nolinkurl{10.1007/s00454-024-00713-2}}.

\bibitem{Simon26}
D. G. Simon, Further analysis of peeling sequences,
\emph{Studia Sci. Math. Hungar.} \textbf{63} (2026), no.~2, 175--190.
doi:\href{https://doi.org/10.1556/012.2026.04354}{\nolinkurl{10.1556/012.2026.04354}}.
Also available as \href{https://arxiv.org/abs/2510.03832}{arXiv:2510.03832}.


\end{thebibliography}
\end{document}